\documentclass[12pt,a4paper,oneside]{article}
\usepackage[english]{babel}
\usepackage{t1enc}
\usepackage{makeidx}
\makeindex
\usepackage{appendix}
\usepackage{amsfonts}
\usepackage[all,cmtip]{xy}
\usepackage{amsthm}
\usepackage{amssymb}
\usepackage{color,soul}
\usepackage{framed}
\usepackage{fancyhdr}
\usepackage{setspace}
\usepackage{mathtools}
\usepackage{wrapfig}
\usepackage{extarrows}
\usepackage{amsmath}
\usepackage{mathrsfs}
\usepackage{caption}
\usepackage[utf8]{inputenc}
\usepackage{graphicx}
\newtheorem{maintheorem}{Theorem}

\renewcommand{\d}{\displaystyle}
\renewcommand{\r}{\Rightarrow}

\renewcommand{\b}{\textbf}

\newcommand{\Div}{\operatorname{div}}

\newcommand{\graf}{\operatorname{Graph}}
\newcommand{\Per}{\operatorname{Per}}
\newcommand{\sr}{\mathbb{R}}

\newcommand{\grad}{\operatorname{grad}}

\newcommand{\id}{\operatorname{I}}

\newcommand{\idd}{\operatorname{id}}
\newcommand{\diff}{\operatorname{Diff}}

\makeatletter
\newcommand{\tpitchfork}{%
  \vbox{
    \baselineskip\z@skip
    \lineskip-.52ex
    \lineskiplimit\maxdimen
    \m@th
    \ialign{##\crcr\hidewidth\smash{$-$}\hidewidth\crcr$\pitchfork$\crcr}
  }%
}
\newtheorem{teo}{Theorem}[section]
\newtheorem{cor}{Corollary}[section]
\newtheorem{lema}{Lemma}[section]
\newtheorem{prop}{Proposition}[section]

\theoremstyle{definition}

\begin{document}

\title{On the integrability of codimension $1$ invariant subbundles of partially hyperbolic skew-products}
\author{Vanderlei Horita \and Ricardo Chical\'{e} Lemes}
\date{}
\maketitle

\let\thefootnote\relax
\let\thefootnote\relax
\let\thefootnote\relax\footnote{Work partially supported by CAPES, CNPq, FAPESP, and PRONEX.}

\begin{abstract}
We prove there is a class of maps $\gamma:\mathbb{T}^{2n}\rightarrow\mathbb{S}^1$ such that a conservative dynamically coherent partially hyperbolic skew-product on $\mathbb{T}^{2n}\times\mathbb{S}^1$ with fixed hyperbolic dynamics on the base and rotation by angle $\gamma$ acting on the fibers have integrable hyperbolic structure which also implies in particular that they are not contact diffeomorphisms. In dimension $3$, we prove the same result using a standard technique in Contact Geometry, namely, that of \emph{characteristic foliations}, which gives a simple proof of the result but with more tight restrcitions to the map $\gamma$.  	
	
\end{abstract}

\section{Introduction}

The goal of this work is to contribute in the study of partially hyperbolic contact diffeomorphisms. Most of the works related to contact dynamics is about \emph{Anosov contact flows}, which are Anosov flows defined in contact manifolds such that its hyperbolic structure $E^s\oplus E^u$ is the contact structure of the manifold~\cite{alv14,alv15,sto15,sto17}.

In general, when the invariant center bundle of a partially hyperbolic diffeomorphism has dimension $1$, the bundles $E^s$ and $E^u$ are not jointly integrable, that is, $E^s\oplus E^u$ is not integrable. However, nonitegrability of $E^s\oplus E^u$ does not mean that it qualifies to be a contact structure. Indeed, a contact struture is defined to be a codimension $1$ subbundle of an odd-dimensional manifold that is \emph{as far as possible from being integrable} or \emph{maximally nonintegrable}. The geometric meaning of \emph{maximal nonintegrability} is that there are no hypersurfaces tangent to the given subbundle, not even locally, whereas an ordinary nonintegrable subbundle may have some form of local integrability. This suggests that we can establish different types of nonintegrability.

A way to establish the different types of nonintegrability is through the \emph{Frobenius' Integrability Theorem}. For codimension $1$ subbundles, the Frobenius' Integrability Theorem states that such subbundles are integrable if and only if $\alpha\wedge d\alpha = 0$, where $\alpha$ is local defining $1$-form for the subbundle.
This theorem implies that one can achieve nonitegrability by providing a subbundle such that $(\alpha\wedge d\alpha)(p)\neq 0$, for some $p\in M$. On the other hand, the contact condition for $\xi$ requires that $\alpha\wedge(d\alpha)^n \neq 0$ at every point in $M$, that is, $\alpha\wedge (d\alpha)^n$ is required to be a volume form in $M$. When $\dim(M) =3$ it becomes clear what ``\emph{as far as possible from being integrable}'' means: $\xi$ is a contact structure if $\alpha\wedge d\alpha \neq 0$ at every point in $M$, which is the complete opposite of it being (Frobenius) integrable.


A \emph{partially hyperbolic contact $C^r$ diffeomorphism} $f:M\rightarrow M$ as we consider here is an element of the set of partially hyperbolic $C^r$ diffeomorphisms $\mathcal{PH}^r(M)$ that preserves a contact form up to multiplication by strictly positive $C^\infty$ functions. Denote by $\diff^r_\alpha(M)$ the set of diffeomorphisms of $M$ such that $f^\ast \alpha  = \tau \alpha$, for some $C^\infty$ function $\tau:M\rightarrow (0,+\infty)$. If $\alpha$ is a contact structure, then $f$ is a partially hyperbolic contact $C^r$ diffeomorphism if $f\in \mathcal{PH}^r(M)\cap\diff^r_\alpha(M)$. This means that the codimension $1$ subbundle $\ker(\alpha)\subset TM$ and the corresponding Reeb vector field $R$ are $f$-invariant, which gives two distinct decompositions of $TM$ into $Df$-invariant subbundles: the usual one, $TM = E^s\oplus E^c \oplus E^u$, inherited from the partially hyperbolic structure and $TM = \ker(\alpha)\oplus \langle R\rangle$ inherited from the contact structure. It follows that either one of the joint subbundles $E^{cs} = E^s\oplus E^c$, $E^{cu} = E^c\oplus E^u$ or $E^{su} = E^s\oplus E^u$ are subbundles of the contact structure $\ker(\alpha)$. This poses strict dimensional restrictions to either $E^s$ or $E^u$, whenever they belong to $\ker(\alpha)$. Indeed, since these subbundles are integrable they must satisfy $$\dim(E^\delta) \leq \dfrac{1}{2} \dim(\ker(\alpha)),$$ for $\delta = s$ or $u$. (See Theorem~\ref{contsub}).

For dynamically coherent partially hyperbolic contact diffeomorphisms, the only possibility for the invariant contact structure is that it contains $E^s\oplus E^u$ and when the central bundle has dimension $1$ this is actually the unique choice, i.e., $E^s\oplus E^u$ is a contact structure for any dynamically coherent partially hyperbolic contact diffeomorphisms. Note that in this case, $E^s\oplus E^u$ is at least $C^1$, which shows that dynamically coherent partially hyperbolic contact diffeomorphisms does not occur frenquently since $E^s\oplus E^u$ is generally at most Hölder continuous.  This poses the question of whether such diffeomorphisms actually exist or not on manifolds that supports them. For instance, in the Heisenberg $3$-manifold, which is the quotient of the Heisenberg group by a discrete subgroup and can also be viewed as a $\mathbb{S}^1$-bundle of over $\mathbb{T}^2$, all bundle isomorphisms over Anosov automorphisms of $\mathbb{T}^2$ are contact diffeomorphisms with $E^s\oplus E^u$ being the invariant contact structures in Heisenberg manifold, see~\cite{shi14}. On the other hand we have yet to find an example of dynamically coherent partially hyperbolic contact diffeomorphism on $\mathbb{T}^{3}$. 

The existence of dynamically coherent partially hyperbolic contact diffeomorphisms in $\mathbb{T}^{2n+1}$ with one dimensional center bundle is an interesting matter since the contact structure must be $E^s\oplus E^u$. In particular, if there are any contact partially hyperbolic diffeomorphisms in $\mathbb{T}^{2n+1}$ with $C^\infty$ contact structure $E^s\oplus E^u$ then it must have the \emph{accessibility property} due to \emph{Chow's Theorem}, see \cite[Theorem 3.3]{pgshb93}.

In this work we consider mainly conservative partially hyperbolic skew-products $F:\mathbb{T}^{2n}\times\mathbb{S}^1\rightarrow \mathbb{T}^{2n}\times \mathbb{S}^1$ of the form
\begin{equation}
F(x,t) = (f(x),t+\gamma(x)),
\label{eq000}
\end{equation}
where $f$ is a symplectic Anosov $C^r$ diffeomorphism and $\gamma\in C^r(\mathbb{T}^{2n},\mathbb{S}^1)$. This choice stems from the fact that most partially hyperbolic diffeomorphisms with one dimensional center bundle lies in the isotopy class of $f\times \idd_{\mathbb{S}^1}$ are skew-products, where $f:\mathbb{T}^{2n}\rightarrow \mathbb{T}^{2n}$ is Anosov. 

	
		
		
		

	

We show that for a certain class of maps $\gamma:\mathbb{T}^{2n}\rightarrow\mathbb{S}^1$, namely \emph{coboundaries} (see Section 2 for precise definitions), the skew-product~\eqref{eq000} fails to be a contact diffeomorphism.

\begin{maintheorem}
	There are no partially hyperbolic contact $C^r$ skew-products $F:\mathbb{T}^{2n}\times \mathbb{S}^1\rightarrow \mathbb{T}^{2n}\times \mathbb{S}^1$ of the form $F(p,t) = (f(p),t+\gamma(p))$, where $f:\mathbb{T}^{2n}\rightarrow \mathbb{T}^{2n}$ is an Anosov symplectic $C^r$ diffeomorphism and $\gamma\in C^r(\mathbb{T}^{2n};\mathbb{S}^1)$ is a coboundary, for $r\geq 1$.
	\label{teoA}
\end{maintheorem}
 
While this does not prove the existence of partially hyperbolic contact diffeomorphisms in $\mathbb{T}^{2n+1}$, it does exclude an important class of maps $\gamma$ from the ones that can produce such diffeomorphisms and are tightly related to our cadidate subbundle for contact structure to admit a maximal integral smooth manifold. This is the content of the following theorem, which is also the main tool in proving Theorem~\ref{teoA}.

\begin{teo}
	
	Let $M^n$ be a compact $n$-dimensional manifold and $F:M\times \mathbb{S}^1\rightarrow M\times \mathbb{S}^1$ be a partially hyperbolic $C^r$ skew-product of the form $F(p,t) = (f(x),t+\gamma(p))$, where $f$ is a Anosov  $C^r$ diffeomorphism  and $\gamma\in C^r(M;\mathbb{S}^1)$, $r\geq 1$. Then $\gamma$ is a coboundary relative to $f$ with transfer map $\mu\in C^1(M,\mathbb{S}^1)$ if and only if there exists an $F$-invariant codimension $1$ foliation $\mathcal{F}$ transversal to $\mathbb{S}^1$ such that each leaf is at least $C^1$ fixed by $F$. 
	\label{teoB}
\end{teo}

In the course of the proof of Theorem~\ref{teoB} we show that the skew-product admit an \emph{invariant graph}, which in turn extend to an invariant codimension $1$ foliation transversal to the fibers. An \emph{invariant graph} for a skew-product $F:M\times\mathbb{S}^1\rightarrow M\times\mathbb{S}^1$, given by 
\begin{equation}
F(x,t) = (f(x),g_x(t)),
\label{eq0000}
\end{equation} 
is a map $\mu:M\rightarrow \mathbb{S}^1$ such that 
\begin{equation}
\mu(f(x)) = g_x(\mu(x))
\label{eq0001}
\end{equation}

Hadjiloucas, Nicol and Walkden proved in~\cite{had01} that for skew-products~\eqref{eq0000}, with base Anosov dynamics and negative fiber Lyapunov exponent, a dichotomy occurs: either the invariant graph is nowhere differentiable or is as smooth as the dynamics, with the latter case generically ocurring. More recently, Walkden and Withers studied the case of a skew-product~\eqref{eq0000} defined on $M\times\mathbb{R}$ over expanding Markov maps and with the fiber Lyapunov exponents vanishing on a set of periodic orbits. In this case, he also obtained a dichotomy where either the invariant graph has the structure of a \emph{quasi-graph} or is smooth as the dynamics~\cite{walk18}.


In dimension $3$ we can use techniques from Contact Topology, namely \emph{characteristic foliations}, to prove Theorem~\ref{teoA} in particular case when the transfer map of the couboundary is $C^r$, with $r\geq 2$. 

The following result is the base of the proof.

\begin{teo}
	Let $M^{2}$ be a compact surface without boundary, $F:M\times \mathbb{S}^1\rightarrow M\times \mathbb{S}^1$ be a skew-product of the form $F(p,t) = (f(x),t+\gamma(p))$, where $f$ is a conservative Anosov  $C^r$ diffeomorphism  and $\gamma\in C^r(M;\mathbb{S}^1)$, $r\geq 1$. If $\gamma$ is a coboundary relative to $f$ with transfer map $\mu\in C^1(M;\mathbb{S}^1)$ then there exists an $DF$-invariant codimension $1$ subbundle $\xi\subset T(M\times\mathbb{S}^1)$  transversal to $T\mathbb{S}^1$ such that the vector field $X_\xi$ reprensenting the characteristic foliation of $\xi$ in any hypersurface $S\subset M\times \mathbb{S}^{1}$ transversal to $T\mathbb{S}^1$ is locally hamiltonian. 
	\label{teoC}
\end{teo}

Note that in Theorem~\ref{teoC}, the resulting locally hamiltonian vector field can be only continuous and lacking differentiability. The restriction on the transfer map to be at least $C^2$ in Theorem~\ref{teoA'} is to guarantee the minimum  differentiability required of the characteristic foliation so that we can apply the techniques of Contact Topology.

\section{Basic definitions}

A diffeomorphism $f:M\rightarrow M$ is said to be \emph{partially hyperbolic} if there exists a $Df$-invariant splitting $TM = E^s\oplus E^c\oplus E^u$ of the tangent bundle satisfying the following conditions:
\begin{enumerate}
	\item[(i)] $Df(E^\delta) = E^\delta$, para $\delta = s, c$ e $u$;
	
	\item[(ii)] there are constants $C>0$ e $0<\lambda<1$ such that $$||Df_p^nv_s||\leq C\lambda^n||v_s|| \quad \mbox{e} \quad ||Df_p^{-n}v_u||\leq C\lambda^n ||v_u||,$$ for all $p\in M$, $v_s\in E^s_p$ and $v_u\in E^u_p$;
	\item[(iii)] there are real numbers $0< \mu_1 <1<\mu_2$, usually dependent on the point $p$, such that  $$||Df_p|_{E_p}||<\mu_1 <  ||Df_p|_{E^c_p}||  <\mu_2 <||Df_p|_{E^u_p}||,$$ where $||Df_p|_{E^\delta_p}||$ is the norm of the linear transformation $Df_p|_{E^\delta_p}: E^\delta_p\rightarrow E^\delta_{f(p)}$. When $\mu_1$ and $\mu_2$ are independent of $p$ we say that $f$ is \emph{absolutely partially hyperbolic}.
\end{enumerate}

A \emph{contact manifold} $(M,\xi)$ is an odd dimensional manifold with a codimension $1$ subbundle $\xi$ that is as far as possible from being integrable in the sense of Frobenius. This means that if $\alpha$ is a local defining $1$-form for $\xi$, i.e., $\xi|_U = \ker(\alpha)|_U$ is some open neighborhood $U\subset M$, then $\alpha\wedge (d\alpha)^n \neq 0$ at every point. Such a bundle is called a \emph{contact structure} in $M$. When a globally defined $1$-form $\alpha$ satisfies the condition $\alpha\wedge(d\alpha)^n\neq 0$ at every point in $M$, we say that $\alpha$ is \emph{contact form} in $M$ and $\xi = \ker(\alpha)$ is the associated contact structure. The condition $\alpha\wedge(d\alpha)^n\neq 0$ is equivalent to $d\alpha|_{\ker(\alpha)}$ being nondegenerate which implies that a contact struture arising from a contact form is a particular case of a \emph{symplectic vector bundle}. The \emph{Reeb vector field} associated to a contact form $\alpha$ is the unique vector field $R$ satisfying the following conditions:

\begin{equation}
\left\{
\begin{array}{l}
\alpha(R) = 1\\
\iota_Rd\alpha \equiv 0
\end{array}
\right.
\label{reeb}
\end{equation}

The condition $\iota_Rd\alpha \equiv 0$ means that for every vector field $Y$ in $M$ we have $d\alpha(R,Y) = 0$. Since $d\alpha|_{\ker(\alpha)}$ is nondegenerate, this condition implies that $R$ is one of the vector fields in the unique degenerate direction of $d\alpha$ in $TM$ and the condition $\alpha(R) = 1$ allow us to pick the unique vector field in that direction that is normalized by $\alpha$.

A diffeomorphism $f:M\rightarrow M$ is a \emph{contact diffeomorphism} if there exist a $C^\infty$ function $\lambda: M\rightarrow \mathbb{R}^{+}$ such that $f^\ast\alpha = \lambda\alpha$. This condition implies that both $\xi$ and $\langle R\rangle$ are $Df$-invariant subbundles, where $R$ is the Reeb vector field associated to $\alpha$.

An \emph{isotropic submanifold} of a contact manifold $(M^{2n+1},\alpha)$ is a submanifold $N\subset M$ such that $T_pN \subset \ker(\alpha)_p$, for all $p\in N$. Since $\ker(\alpha)$ is not Frobenius integrable, i.e., $\ker(\alpha)$ does not admit integral submanifolds of maximal dimension, we must have $\dim(N)< 2n$. In fact, the following holds:

\begin{teo}
	Let $(M^{2n+1},\alpha)$ be a contact manifold and $N\subset M$ be an isotropic submanifold of $M$. Then, $\dim(N)\leq n$. Moreover, contact diffeomorphisms maps isotropic submanifolds to isotropic submanifolds.
	\label{contsub}
\end{teo}  

A proof of this result and more details regarding the structure of submanifolds of contact manifolds can be found in~\cite{geig08}.

Let $G$ be a Lie group and consider the set $\Lambda^r_{k}(M,G)$ of $G$-valued $C^r$ differential $k$-forms of $M$ and let $f:M\rightarrow M$ be a diffeomorphism. A differential $k$-form $\alpha\in \Lambda^r_{k}(M,G)$ is a \emph{coboundary with respect to} a diffeomorphism $f:M\rightarrow M$ if there exists a $C^\ell$ differential $k$-form $\beta$, with $\ell \leq r$, such that 
\begin{equation}
f^\ast\beta - \beta = \alpha.
\label{eq2.1}
\end{equation}
The cases when $k=0$ and $1$, in which $\Lambda^r_0(M,G) = C^r(M;G)$ and $\Lambda^r_1(M) = \Omega^r(M,G)$, are the modules we deal with in this work. 

An equation of the form~\eqref{eq2.1} where both $f$ and $\alpha$ are known and we want to obtain $\beta$ is called a \emph{cohomological equation}. These equations have been extensively studied in the case $k=0$ for different types of groups and more recently by Avila and Kocsard for general diffeomorphism cocycles~\cite{avila18}. Results that guarantee the existence of solutions to the various types cohomological equation under certain conditions are sometimes collectively called \emph{Liv\u{s}ic's Theorems} due to the following result of the seminal work of Liv\u{s}ic [ref]:

\begin{teo}[Liv\u{s}ic's Theorem]
	Let $f:M\rightarrow M$ be an Anosov diffeomorphism of a compact manifold $M$, $(G,+)$ be a commutative Lie group and $\gamma\in C^{\theta}(M; G)$ be a $\theta$-Hölder continuous map. If for each $n$ we have $$\sum_{j=0}^{n-1} \gamma(f^{j}(x)) = 0,$$ for every $n$-periodic point $x$ of $f$, then there exist $\mu\in C^\theta(M;G)$ such that $\gamma = \mu\circ f -\mu$.
\end{teo}

Let $S\subset M^3$ be a surface in a $3$-dimensional manifold and $\xi\subset TM$ be codimension $1$ subbundle. The \emph{characteristic foliation} of $\xi$ in $S$ is the foliation of $S$ generated by the field of directions $\xi\cap TS$. A field of directions can be thought as an equivalence class of vector fields modulo multiplication by $C^\infty$ strictly positive functions and a representative of this equivalence class is what we call a \emph{characteristic vector field}. In general, the characteristic foliation has singular points which are precisely the point where $\xi$ is tangent to $S$.

Let $(M,\omega)$ be a symplectic manifold and $g$ a riemannian metric in $M$. An \emph{almost complex structure} in $M$ is a $(1,1)$-tensor $\b{J}:TM\rightarrow TM$ satisfying $\b{J}^2 = - \id_{2n}$. The almost complex structure $\b{J}$ is said to be \emph{compatible} with the symplectic and the riemmanian structures of $M$ if the following conditions holds:
\begin{enumerate}
	\item[(i)] $\omega(u,\b{J}v) = g(u,v)$;
	\item[(ii)] $g(\b{J}u,v) = \omega(u,v)$,
	\item[(iii)]  $\b{J}u = S^{-1}\mu(u)$,
\end{enumerate}
where $S:TM\rightarrow TM$ and $\mu:TM\rightarrow TM$ are the isomorphisms induced by $g$ and $\omega$, respectively. In this case we also say that $(\omega,g,\b{J})$ is a \emph{compatible triple}. All symplectic manifolds admits a compatible almost complex structure (see Proposition 4.1 in~\cite{McDi98}).

Let $H:M\rightarrow\mathbb{R}$ be a $C^k$ function on an symplectic almost complex manifold $(M,\omega,g,\b{J})$ with a compatible triple of structures. A \emph{hamiltonian vector field} $X_H$ in $M$ is the vector field defined by $$X_H = -\b{J}\grad(H),$$ where $\grad$ is the gradient of the function $H$ with respect to the metric $g$.

\section{Proof of Theorem A}





\subsection{Main proof}


	


	

For a partially hyperbolic diffeomorphism on a $3$-dimensional manifold the only codimension $1$ invariant subbundles are $E^s\oplus E^u$, $E^s\oplus E^c$ or $E^c\oplus E^u$. Since the skew-product in Theorem A is dynamically coherent it follows that the only possibility for an invariant contact structure is $E^s\oplus E^u$.

The $DF$-invariant central direction is the direction of the fibers, i.e., $E^c = T\mathbb{S}^1$ and since $\gamma$ is a coboundary, Theorem~\ref{teoB} implies that there exist a codimension $1$ subbundle $\xi$ transversal to $T\mathbb{S}^1 = E^c$ such that $\xi$ is tangent to the graph of the transfer map $\mu\in C^1(\mathbb{T}^{2n+1};\mathbb{S}^1)$ of $\gamma$. However, the only $DF$-invariant subbundle transversal to $T\mathbb{S}^1$ is $E^s\oplus E^u$. Therefore, $E^s\oplus E^u$ is tangent to the graph of $\mu$ and Theorem A follows.


\subsection{Particular case in dimension $3$}

\begin{teo}[Theorem A]
	There are no partially hyperbolic contact $C^r$ skew-products $F:\mathbb{T}^{2}\times \mathbb{S}^1\rightarrow \mathbb{T}^{2}\times \mathbb{S}^1$ of the form $F(p,t) = (f(p),t+\gamma(p))$, where $f:\mathbb{T}^{2}\rightarrow \mathbb{T}^{2}$ is an Anosov conservative $C^r$ diffeomorphism and $\gamma\in C^r(\mathbb{T}^{2};\mathbb{S}^1)$ is a coboundary with $C^\ell$ transfer map, for $2\leq \ell\leq r$.
	\label{teoA'}
\end{teo}

\begin{proof}

Since $\gamma$ is a coboundary, Theorem~\ref{teoC} implies that there exist a codimension $1$ subbundle $\xi$ such that for any surface $S$ transversal to $\mathbb{S}^1$ the characteristic foliation can represented by a locally hamiltonian vector field $X_\xi = -\b{J}\grad(H)$, for some differentiable function $H: M\rightarrow\mathbb{S}^1$.

The following lemma gives a necessary and sufficient condition for a codimension $1$ subbundle to be a contact structure near an embedded surface in a $3$-dimensional manifold.

\begin{lema}
A vector field $X$ on an embedded surface $S$ in a $3$-dimensional manifold represents the characteristic foliation of a contact structure if and only if the following condition holds: $\Div_\omega(X)(p) \neq 0$ on all points singular points $p\in S$ of $X$, where $\Div_\omega(X)$ is the divergence of the vector field $X$ with respect an area form $\omega$ in $S$
\label{lema31}
\end{lema}

For a proof of Lemma~\ref{lema31}, see Lemma 4.6.3 in~\cite{geig08}.

Recall that the divergence of vector field $X$ with respect to some area form $\omega$ in $S$ is uniquely defined the function  $\Div_\omega(X): S\rightarrow \mathbb{R}$ such that $$\mathcal{L}_X\omega = \Div_\omega(X)\omega,$$ where $\mathcal{L}_X\omega$ is the Lie derivative of $\omega$ with respect to $X$. By Cartan's Formula, $$\mathcal{L}_X\omega = \iota_Xd\omega+d(\iota_X\omega),$$ and the fact $\omega$ is a symplectic form in $S$, we have $$\Div_\omega(X)\omega = d(\iota_X\omega).$$

Now, choosing Darboux local (symplectic) coordinates on $S$ we write $\omega = dx\wedge dy$ and $$X = -\dfrac{\partial H}{\partial y}\dfrac{\partial}{\partial x}+\dfrac{\partial H}{\partial x}\dfrac{\partial}{\partial y}$$ in these coordinates. It follows that $\iota_X\omega = dH$ and since $\omega$ is nondegenerate, we conclude that $\Div_\omega(X)(p)\omega = 0$ at any  point. This implies that the $DF$-invariant codimension $1$ subbundle $\xi$ of Theorem~\ref{teoC} cannot be a contact structure.

Again, since the only $DF$-invariant codimension $1$ subbundle transversal to $T\mathbb{S}^1$ is $E^s\oplus E^u$ it follows that $F$ cannot be a contact diffeomorphism and Theorem A follows.

\end{proof}

\section{Proof of Theorem 1.1}
Let $\alpha$ be a $1$-form in $M\times \mathbb{S}^1$, transversal to $\mathbb{S}^1$ in the sense that the tangent direction $\dfrac{\partial}{\partial t}$ to $\mathbb{S}^1$ is transversal to $\ker(\alpha)$, for all $(p,t)\in M\times\mathbb{S}^1$, and $F:M\times\mathbb{S}^1\rightarrow M\times \mathbb{S}^1$ a skew-product of the form $$F(p,t) = (f(p),\phi_t(p)),$$ where $f:M\rightarrow M$ is an Anosov diffeomorphism and $\phi(p):\mathbb{S}^1\rightarrow \mathbb{S}^1	$ a family of circle diffeomorphisms with parameters in $M$.



\begin{lema}
	Let $\alpha = v_tdt+\beta_t$ be a $1$-form in $M\times \mathbb{S}^1$, where $\beta_1\in\Omega^\ell_1(M)$, and suppose that $\ker(\alpha)$  is $F$-invariant and transversal to $\mathbb{S}^1$. Then,
	\begin{equation}
	F^\ast\left(\dfrac{\beta_{t}}{v_t}\right) - \phi_t'\dfrac{\beta_t}{v_t} = -d\phi_t,
	\label{lemaeq31}
	\end{equation}
	where $\phi_t'$ is the derivative of $\phi_t$ with respect to $t$ and $F^\ast: \Omega^\ell_1(M\times\mathbb{S}^1)\rightarrow \Omega^\ell_1(M\times\mathbb{S}^1)$ is the pull-back transformation induced by $F$.
	Conversely, if~\eqref{lemaeq31} holds, with $v_t\neq 0$ at every point, then $\ker(\alpha)$ $DF$-invariant.
	\label{lema3121}
\end{lema}

\begin{proof}
	
	Consider the splitting $T(M\times \mathbb{S}^1)\cong TM\oplus T\mathbb{S}^1$ induced by the global trivialization of the trivial bundle $M\times\mathbb{S}^1$. This splitting induces a module isomorphism between the $C^r(M\times\mathbb{S}^1,\mathbb{R})$-modules $\Omega^r_1(M\times\mathbb{S}^1)$ and $\Omega^r_1(M)\oplus\Omega^r_1(\mathbb{S}^1)$, which is realized by taking the natural projections $\pi_1: TM\oplus T\mathbb{S}^1\rightarrow TM$ and $\pi_2:TM\oplus T\mathbb{S}^1\rightarrow T\mathbb{S}^1$ and noting that the induced pullback maps $\pi_1^\ast:\Omega^r_1(M)\rightarrow \Omega^r_1(M\times\mathbb{S}^1)$ and $\pi_2^\ast: \Omega^r_1(\mathbb{S}^1)\rightarrow \Omega^r_1(M\times\mathbb{S}^1)$ satisfies $\Omega^r_1(M\times\mathbb{S}^1) = \pi_1^\ast\Omega^r_1(M)\oplus \pi_2^\ast\Omega^r_1(\mathbb{S}^1)$. Hence any $1$-form in $\Omega_1^r(M)$ or $\Omega^r_1(\mathbb{S}^1)$ can be considered as a $1$-form in $\Omega_1^r(M\times \mathbb{S}^1)$ via this identification.
	
	Since $\ker(\alpha)$ is $DF$-invariant there exists $\lambda: M\times \mathbb{S}^1\rightarrow \mathbb{R}$ such that $F^\ast \alpha = \lambda\alpha$. Then $$F^\ast\alpha = \lambda\alpha \r F^\ast(v_tdt +\beta_t) = \lambda(v_tdt+\beta_t) \r $$
	
	\begin{equation}
	F^\ast(v_tdt)+F^\ast\beta_t = \lambda v_tdt+\lambda\beta_t
	\label{peq1}
	\end{equation}

	Note that $$F^\ast(v_tdt) = (v_{\phi_t}\circ f) d\phi_t = (v_{\phi_t}\circ f) d\phi_t+(v_{\phi_t}\circ f)\phi_t'dt,$$ where $d\phi_t$ is the partial derivative of $\phi_t$ in the direction of $TM$ with respect to the splitting $T(M\times\mathbb{S}^1) \cong TM\oplus T\mathbb{S}^1$. For simplicity, we will write $v_{\phi_t}$ instead of $v_{\phi_t}\circ f$ so that $F^\ast(v_tdt) = v_{\phi_t}d\phi_t+v_{\phi_t}\phi'_tdt$. Then, from this and~\eqref{peq1} we have
	
	$$(v_{\phi_t}\phi'_t -\lambda v_t)dt = F^\ast\beta_t - \lambda \beta_t +v_{\phi_t}d\phi_t.$$
	
	Note that the left side of this equation is a multiple of $dt$ while that right side is a linear combination of $dx$ and $dy$, since $\beta_t\in\Omega^\ell_1(M)$. Now, since $\left\{dx, dy, dt\right\}$ is a local basis for $\Omega^\ell_1(M\times\mathbb{S}^1)$, it follows that
	
	\begin{equation}
	\phi'_tv_{\phi_t} - \lambda v_t = 0
	\label{peq2}
	\end{equation}
	and
	\begin{equation}
	F^\ast\beta_t - \lambda \beta_t+v_{\phi_t}d\phi_t = 0
	\label{peq3}
	\end{equation}
	by linear independency of $dx$, $dy$ and $dt$. The transversality condition of $\ker(\alpha)$ implies that either $v_t>0$ or $v_t<0$ at every point, otherwise there would be point $(p,t)\in M\times\mathbb{S}^1$ such that $T_p\mathbb{S}\subset \ker(\alpha)_{(p,t)}$. Thus, $\dfrac{\beta_t}{v_t}$ is defined in all of $M\times\mathbb{S}^1$ and by~\eqref{peq3} we obtain 
	$$-d\phi_t = \dfrac{F^\ast\beta_t}{v_{\phi_t}} - \dfrac{\lambda \beta_t}{v_{\phi_t}}$$
	
	Recall that $F^\ast v_t = v_{\phi_t}$, which implies $F^\ast (v_t)^{-1} = v_{\phi_t}^{-1}$. Also,  by~\eqref{peq2} we have $\dfrac{\lambda}{v_{\phi_t}} = \dfrac{\phi'_t}{v_t}$. Then, by the previous equation we obtain
	
	$$-d\phi_t =   F^\ast\left(\dfrac{\beta_t}{v_t}\right) - \phi'_t\dfrac{\beta_t}{v_{t}}.$$
	


	Conversely, suppose~\eqref{lemaeq31} holds. Then, $F^\ast\left(\dfrac{\beta_t}{v_t}\right) =-d\phi_t   + \phi'_t\dfrac{\beta_t}{v_{t}}$ and $$F^\ast \left(\dfrac{\alpha}{v_t}\right) = F^\ast(dt)+F^\ast\left(\dfrac{\beta_t}{v_t}\right) =$$ $$ = d\phi_t+\phi_t'dt+\phi_t'\dfrac{\beta_t}{v_t}-d\phi_t = \phi_t'dt+\phi_t'\dfrac{\beta_t}{v_t} \r $$ $$  F^\ast(v_t^{-1})F^\ast\alpha = \dfrac{\phi'_t}{v_t}\left(v_tdt+\beta_t\right) \r$$ $$(v_{\phi_t}\circ f)^{-1} F^\ast \alpha = \dfrac{\phi'_t}{v_t}\left(v_tdt+\beta_t\right) \r$$  $$F^\ast \alpha = \dfrac{\phi'_t v_{\phi_t}\circ f}{v_t}\left(v_tdt+\beta_t\right) = \lambda \alpha$$ where $\lambda = v_t^{-1}\cdot\phi_t'\cdot(v_{\phi_t}\circ f)$. Therefore, $\ker(\alpha)$ is $F$-invariant.
	
\end{proof}

The previous lemma assumes a differential $1$-form with a specific form, which at first seems to be very restrictive. The next lemma however, shows that this is not really the case.

\begin{lema}
	Let $M^n$ be a compact riemannian manifold and $S\subset M$ be a codimension $1$ closed submanifold of $M$. Then, every $C^r$ differential $1$-form $\alpha$ can be written as $\alpha = v_tdt+\beta_t$ in a neighborhood $V_\varepsilon$ of $S$, where $v_t\in C^r(S;\mathbb{R})$ and $\beta_t\in \Omega^r_1(S)$ are families of $C^r$ maps and families of $C^r$ differential $1$-forms of $S$, respectively.
	\label{lema21}	
\end{lema}

\begin{proof}
	
	Consider a tubular neighborhood $V$ of $S$ in $M$. Since $M$ is compact we can find $\varepsilon>0$ such that $S\times(-\varepsilon,\varepsilon)$ is embedded into $V$, where we identify $S$ with $S\times\left\{0\right\}$.
	
	
	Let $\varphi: S\times(-\varepsilon,\varepsilon)\rightarrow V$ be this embedding and $\widetilde{V}_\varepsilon = \varphi(S\times(-\varepsilon,\varepsilon))$. Consider the projection $\pi: \widetilde{V}_\varepsilon \rightarrow S$ induced by the canonical projection $\pi_0:S\times(-\varepsilon,\varepsilon)\rightarrow S$ defined by $\pi = \pi_0\circ\varphi^{-1}$. For $\varepsilon>0$ sufficiently small we can assume that $\widetilde{V}_\varepsilon$ belongs to a cover of $S$ by slice charts of $M$, i.e., $$S\subset \bigcup_{i\in I} W_i,$$ where $(W_i,x_1,\ldots,x_{n-1},t)$ are local coordinates with the property that   \linebreak $S\cap W_i = \left\{p\in W_i; \ t(p)=0 \right\}$. Note that we can always find such an $\varepsilon$ by passing to a finite subcover of $S$ and take $\varepsilon>0$ to be the smallest $\varepsilon_i>0$ such that $\widetilde{V}_{\varepsilon_i} = \varphi((S\cap W_i)\times(-\varepsilon_i,\varepsilon_i))\subset W_i$. Then, if $\alpha$ is a $1$-form in $M$ its restriction to the local charts $\widetilde{V}_{\varepsilon_i}$ is written as $$i_{\widetilde{V}_\varepsilon}^\ast\alpha = v_tdt+\sum_{j=1}^{n-1} u_jdx_j = v_tdt+\beta_t,$$ where $i_{\widetilde{V}_\varepsilon}:\widetilde{V}_\varepsilon\rightarrow M$ is the inclusion, $v_t\in C^k(S;\mathbb{R})$ and $\beta_t = \d\sum_{j=1}^{n-1} u_j( \ \cdot \ , t)dx_j \in\Omega_1^k(S)$ is the local representation of $\beta_t$ with coordinate functions $u_j( \ \cdot \ , t)\in C^k(\widetilde{V}_\varepsilon\cap S;\mathbb{R})$.
	
\end{proof}

\vspace{0.5cm}

\pagebreak

\begin{prop}
Let $M^n$ be a compact $n$-dimensional manifold and $F:M\times \mathbb{S}^1\rightarrow M\times \mathbb{S}^1$ be a skew-product of the form $F(p,t) = (f(x),t+\gamma(p))$, where $f$ is a Anosov  $C^r$ diffeomorphism  and $\gamma\in C^r(M;\mathbb{S}^1)$, $r\geq 1$. If $\gamma$ is coboundary with transfer map $\mu\in C^1(M;\mathbb{S}^1)$ then $\graf(\mu)$ is tangent to a $DF$-invariant codimension-$1$ subbundle $\xi\subset T(M\times\mathbb{S}^1)$ transversal to $\mathbb{S}^1$.
\label{teo4.1}	
\end{prop}

\begin{proof}
	
	Let $\xi$ be a $DF$-invariant codimension $1$ subbundle transversal to $T\mathbb{S}^1$ and consider a local defining $1$-form $\alpha$ for $\xi$. Take any point $(p_0,t_0)\in M\times\mathbb{S}^1$ and let $U\subset M\times \mathbb{S}^1$ be a neighborhood of $(p_0,t_0)$ such that $\xi|_U = \ker(i_U^\ast\alpha)$, for some $1$-form $\alpha$, where $i_U:U\rightarrow M\times \mathbb{S}^1$ is the inclusion.
	
	
	By Lemma~\ref{lema21}, there exist a sufficiently small open neighborhood $V_{\varepsilon,t_0}$ of $M\times\left\{ t_0\right\}$ such that $$i_{U\cap V_{\varepsilon,t_0}}^\ast\alpha = v_tdt+\beta_t,$$ where $v_t:U\cap V_{\varepsilon,t_0}\rightarrow\sr$ and $\beta_t\in\Omega^\ell_1(M)$ are $C^\ell$ families with parameters in $\mathbb{S}^1$, with $\ell\leq r$.
	
	
	Since $\xi$ is transversal to $T\mathbb{S}^1$ we have either $v_t>0$ or $v_t<0$. Recall that $$\phi_t = t+\gamma \ (\mbox{mod.} \ 1),$$ for some function $\gamma\in C^k(M;\mathbb{R})$. From the hypothesis of $F$-invariance of $\xi$ and by Lemma~\ref{lema3121} we have
	\begin{equation}
	F^\ast\left(\dfrac{\beta_t}{v_t}\right) - \dfrac{\beta_t}{v_t} = -d\gamma,
	\label{eq3333}
	\end{equation}
	with $v_t$ satisfying  $\lambda v_t= v_{t+\gamma}\circ f$, for some $C^\infty$ function $\lambda: M\times\mathbb{S}^1\rightarrow\sr^{+}$  due to~\eqref{peq2} and the fact that $\phi_t' = 1$. In particular, $d\gamma$ is coboundary relative to $F$.
	
	
	It follows that the $1$-form $\eta = dt+\widetilde{\beta}_t$ is $F$-invariant, where $\widetilde{\beta}_t = v_t^{-1}\beta_t$, since $$F^\ast\eta = d(t+\gamma)+F^\ast\widetilde{\beta}_t = dt+d\gamma+\widetilde{\beta}_t-d\gamma = dt+\widetilde{\beta}_t = \eta,$$ as consequence of~\eqref{eq3333}.

	
	Now suppose that $\gamma$ is coboundary relative to $f$ with transfer function $\mu\in C^1(M;\mathbb{S}^1)$. Then  $$d\gamma = f^\ast(d\mu) - d\mu = F^\ast (d\widetilde{\mu}) - d\widetilde{\mu},$$ where $\widetilde{\mu}:M\times \mathbb{S}^1\rightarrow \mathbb{S}^1$ is defined as $\widetilde{\mu}(x,t) = \mu(x)$.
	
	
	Therefore, $\widetilde{\alpha} = dt-d\widetilde{\mu}$ is also an $F$-invariant $1$-form since $$F^\ast\widetilde{\alpha} = d(t+\gamma) - F^\ast(d\widetilde{\mu}) = dt+F^\ast(d\widetilde{\mu})-d\widetilde{\mu}-F^\ast(d\widetilde{\mu}) = dt-d\widetilde{\mu}= \widetilde{\alpha}.$$ We also note that $\widetilde{\alpha}$ is also transversal to $\mathbb{S}^1$ since $\widetilde{\alpha}\left(\dfrac{\partial}{\partial t}\right) = 1$ and consequently $\dfrac{\partial}{\partial t}\notin \ker(\widetilde{\alpha})$. Now $F$ being partially hyperbolic with one dimensional central direction there is only one $DF$-invariant codimension $1$ subbunble transversal to $E^c = T\mathbb{S}^1$ and hence  $\xi|_{\widetilde{U}} = \ker(\eta|_{\widetilde{U}}) =  \ker(i_{\widetilde{U}}^\ast\widetilde{\alpha})$, where we wrote $\widetilde{U}= U\cap V_{\varepsilon,t_0}$, for simplicity.
	
	Let $S_{\mu} = \graf(\mu)\subset M\times \mathbb{S}^1$ be the graph of $\mu$ and $i_{S_{\mu}}:S_{\mu}\rightarrow M\times \mathbb{S}^1$ be the inclusion. Then, we have $$i_{{S_{\mu}}\cap \widetilde{U}}^\ast\widetilde{\alpha} = d\mu -d\widetilde{\mu} = d\widetilde{\mu}- d\widetilde{\mu} = 0,$$ which implies that $\ker(i_{\widetilde{U}}^\ast\widetilde{\alpha})$ is tangent to $S_\mu$.
	
	Since the choice of the point $(p_0,t_0)$ is arbitrary, it follows that for every $(p,t)\in M\times\mathbb{S}^1$ we can choose a neighborhood $U\subset M$ having the property that $\xi|_{\widetilde{U}}$ is tangent $S_\mu\cap \widetilde{U}$, where $\widetilde{U} = U\cap V_{\varepsilon,t}$. This concludes the proof.
	
	
\end{proof}

	\begin{cor}
	The collection of all rotations $\mu+\theta$ of the map $\mu$ along $\mathbb{S}^1$ generates an $F$-invariant foliation $\mathcal{F} = (A_{\theta})_{\theta\in\mathbb{S}^1}$, where $A_\theta = \graf(\mu+\theta)$, and each leaf $A_\theta$ is fixed for $F$. 	
	\end{cor}
	
	\begin{proof}
	
	Define $\mu_{\theta}:M\rightarrow\mathbb{S}^1$ as $$\mu_{\theta}(p) = {\theta}+\mu(p),$$ for some $\theta\in\mathbb{S}^1$, and $\widetilde{\mu}_{\theta}:M\times \mathbb{S}^1\rightarrow\mathbb{S}^1$ as $\widetilde{\mu}_{\theta}(p,t) = \mu_{\theta}(p)$. Let $S_{\mu_\theta} = \graf(\mu_\theta)\subset M\times \mathbb{S}^1$ be the graph of $\mu_\theta$ and $i_{S_{\mu_\theta}}:S_{\mu_\theta}\rightarrow M\times \mathbb{S}^1$ e the inclusion. Then, we have $$i_{S_{\mu_{\theta}}\cap \widetilde{U}}^\ast\widetilde{\alpha} = d\mu_{\theta} -d\widetilde{\mu}_\theta = d\widetilde{\mu}_\theta- d\widetilde{\mu}_\theta = 0.$$ Hence, $T_{(p,t)}S_{\mu_\theta} = \ker(i^\ast_{\widetilde{U}}\widetilde{\alpha})_{(p,t)}$, for all $(p,t)\in S_{\mu_{\theta}}\cap\widetilde{U}$, which implies $\xi$ is tangent to the graphs of the maps $\mu_{\theta}$, for all $\theta \in\mathbb{S}^1$.
	
	A foliation chart for $\mathcal{F}$ can be easily provided by considering local charts $(U,\b{x})$ for $M$ and $(I,\b{t})$ for $\mathbb{S}^1$, and defining $\varphi: \mathbb{R}^n\times\mathbb{R} \rightarrow U\times I $ as $$\varphi(p_1,\ldots,p_n,\theta) = (\b{x}^{-1}(p_1,\ldots,p_n),\b{t}^{-1}(\theta)+\mu(p)).$$ The image of a plaque $\mathbb{R}^n_{\widetilde{\theta}} =  \mathbb{R}^n\times\left\{\widetilde{\theta} \right\}$ is the set $$\varphi(\mathbb{R}^n_{\widetilde{\theta}}) = \left\{(p,\theta)\in U\times I; \ \theta = \b{t}^{-1}(\widetilde{\theta})+\mu(p)   \right\} = A_{\b{t}^{-1}(\widetilde{\theta})}\cap(U\times I).$$ If $\psi:\mathbb{R}^n\times \mathbb{R}\rightarrow V\times J$ is another foliation chart given by $$\psi(p_1,\ldots,p_n,\theta) = (\b{y}^{-1}(p_1,\ldots,p_n),\b{s}^{-1}(\theta)),$$ with $(U\times I)\cap (V\times J)\neq \emptyset$, then a simple calculation shows that $$\varphi^{-1}\circ \psi(p_1,\ldots,p_n,\theta) = (\b{x}\circ \b{y}^{-1}(p_1,\ldots,p_n),\b{t}(\b{s}^{-1}(\theta)))$$ so that the transition maps are smooth and they map plaques to plaques. 
	
	Now let $F$ be as in Theorem B and consider an arbitrary point $(p,\theta+\mu(p))\in A_{\theta}$. Then $$F(p,\theta+\mu(p)) = (f(p),\theta+\mu(p)+\gamma(p)) = (f(p),\theta+\mu(f(p)))\in A_\theta,$$ since $\gamma(p) = \mu(f(p))-\mu(p)$.	Therefore, $F(A_\theta) = A_\theta$, for all $\theta\in\mathbb{S}^1$.
	
\end{proof}

This proves the ``if'' part in Theorem~\ref{teoB}. To prove the ``only if'' part we will consider the following result in which the converse in Theorem~\ref{teoB} is a particular case.

\begin{prop}
Let $F:M\times\mathbb{S}^1\rightarrow M\times \mathbb{S}^1$ be a skew-product as in Theorem B and suppose $\mathcal{F} = \left\{A_j \right\}_{j\in J}$ is an $F$-invariant continuous foliation of $M\times\mathbb{S}^1$, with leaves that are at least $C^1$ and transversal to the fibers. If for some $j_0\in J$ there is a leaf $A_{j_0}\in\mathcal{F}$, such that $F^k(A_{j_0}) = A_{j_0}$, then $\gamma_k:M\rightarrow \mathbb{S}^1$ given by $$\gamma_k(x) = \sum_{i=0}^{k-1} \gamma(f^{i}(x)),$$ for every $x\in M$, is a coboundary with respect to $f^k$.
\label{propk}  
\end{prop}

\begin{proof}

Suppose $F^k(A_{j_0}) = A_{j_0}$, for some $j_0\in J$ and $k\in\mathbb{N}$. Then, for any $x\in \Per(f)$ such that $f^m(x) = x$, we have $F^{mk}(x,t)= (x,t)$, with $(x,t)\in A_{j_0}$. But $$F^{mk}(x,t) = \left(f^{mk}(x), t+\sum_{i=0}^{mk-1}\gamma(f^i(x)) \right),$$ which implies that 
\begin{equation}
\sum_{i=0}^{mk-1}\gamma(f^i(x)) = 0
\label{eq4.1}
\end{equation}
at the $m$-periodic points of $f$. To finish the proof, we show that if~\eqref{eq4.1} holds then $\gamma_k$ is a coboundary with respect to $f^k$.

Note that $$\sum_{i=0}^{mk-1}\gamma(f^i(x)) = \sum_{i=0}^{k-1}\gamma(f^{i}(x))+\sum_{i=0}^{k-1}\gamma(f^{i+k}(x))+\ldots+ \sum_{i=0}^{k-1}\gamma(f^{i+(m-1)k}(x))$$ $$ =  \sum_{i=0}^{m-1}\gamma_k(f^{ik}(x))$$

It follows from~\eqref{eq4.1} that $\sum_{i=0}^{m-1}\gamma_k(f^{ik}(x)) = 0$ at the $m$-periodic points of $f$, for every $m\in \mathbb{N}$, and consequently, for every $m$-periodic points of $f^k$, since $\Per(f^k)\subset \Per(f)$. Therefore, by Liv\u{s}ic's Theorem there is a $\theta$-Hölder continuous map $\mu\in C^\theta(M;\mathbb{S}^1)$ such that $\gamma_k = \mu\circ f^k- \mu$.\\

\end{proof}

When $k=1$, Proposition~\ref{propk} gives precisely the ``only if'' part  in Theorem~\ref{teoB}.

\section{Proof of Theorem~\ref{teoC}}


To prove Theorem~\ref{teoC} we first give a useful characterization of the characteristic vector fields near the hypersurface $S\subset M$ where they are defined and the local defining $1$-forms of the bundle $\xi$ the characteristic vector fields originate from.

\begin{lema}
	Let $S\subset M$ be a hypersurface and $U\subset M^3$ be an open set such that $U\cap S\neq\emptyset$ and $\xi|_U = \ker(\alpha)$, for some $1$-form $\alpha$. Then, a vector field $X_\xi\in TS$ is a generator of the characteristic foliation if and only if $\iota_{X_\xi}\omega = i_{S}^\ast\alpha$ for some area form $\omega$ in $S$.	
	\label{lema32}
\end{lema}

\begin{proof}
	
	Let $X_\xi\in TS$ be a vector field satisfying $\iota_{X_\xi}\omega = i_S^\ast \alpha$ for some area form $\omega$ in $S$. It suffices to show that $X_\xi\in\ker(\alpha)$ since this implies that $X_\xi\in TS\cap\ker(\alpha)$ and hence is a generator of the characteristic foliation. But this is true since $$i_S^\ast \alpha(X_\xi) = \iota_{X_\xi}\omega(X_\xi) = \omega(X_\xi,X_\xi) = 0.$$
	
	
	Conversely, let $X_\xi$ be a vector field generating the characteristic foliation and $\Omega$ be an arbitrary area form in $S$. Then, the vector field $Y$ satifying $\iota_Y\Omega = i_S^\ast \alpha$ is also a generator of the characteristic foliation and thus, there exists a positive $C^\infty$ function $\lambda:M\rightarrow \mathbb{R}^+$ such that $Y = \lambda X_\xi$.
	
	
	Let $\omega = \lambda\Omega$. Since $\lambda>0$, the $2$-form $\omega$ is also an area form $S$ and  $$\iota_{X_\xi}\omega = \lambda \iota_{X_\xi}\Omega = \iota_{\lambda X_\xi} \Omega = \iota_Y\Omega = i_S^\ast \alpha.$$

\end{proof}


Let $(M^2,\omega)$ be a symplectic manifold,  $F:M\times \mathbb{S}^1\rightarrow M\times \mathbb{S}^1$ be a skew-product as in Theorem A, $M_{t_0} = M\times \left\{t_0 \right\}$ and $\xi\subset TM$ a  codimension $1$ subbundle and consider the slice chart formed by local symplectic coordinates $x,y$ in $M$ and the global coordinate $t$ in $\mathbb{S}^1$. Then, by Lemma~\ref{lema21}, there exists an open set $U\subset M\times \mathbb{S}^1$ such that $U\cap M_t\neq 0$ e $\xi|_{U} = \ker{\alpha}$, with $\alpha = v_tdt+\beta_t$ e $\beta_t\in\Omega^k_1(M_t) = \Omega^k_1(M)$.\\



\begin{lema}
	
	If $X_\xi = X_\xi^1\dfrac{\partial}{\partial x}+X^2_\xi\dfrac{\partial}{\partial y}$ is the characteristic vector field of $\xi$ in local symplectic coordinates of $M_t$ associated to the symplectic form $\omega$, then $\beta_t = -X_\xi^2dx+X_\xi^1dy$.
	\label{prop31}
\end{lema}

\begin{proof}
	
	We first observe that $M_t$ is diffeomorphic to $M$ through $\pi_t:M_t\rightarrow M$ obtained by restricting the projection along the fibers $\pi: M\times \mathbb{S}^1\rightarrow M$ to $M_t$. Indeed, the projection along the fibers of the trivial bundle $M\times\mathbb{S}^1$ is the natural projection on the first component and the graph of any smooth map is diffeomorphic to its domain through the natural projection. Since $M$ is a symplectic manifold with symplectic form $\omega$ we pull back this symplectic form to $M_{t}$ with $\pi_t$ obtaining a symplectic form $\omega_t = \pi_t^\ast\omega$ on each $M_t$. Choosing symplectic local coordinates $x,y$ on $M$ and the global coordinate $t$ on $\mathbb{S}^1$ we have $\omega = dx\wedge dy$ and $\pi_t(x,y,t) = (x,y)$. Then, we trivially have $\omega_t = \pi_t^\ast\omega = \omega$.

	
	The symplectic form $\omega = dx\wedge dy$ is an area form in $M_t$ and by Lemma~\ref{lema32}, the characteristic vector field of $\xi$ in $M_t$ associated to $\omega$ satisfies $\iota_{X_\xi}\omega = i^\ast_{M_t}\alpha$. But, fixing $t_0\in \mathbb{S}^1$, we have $$i_{M_{t_0}}^\ast\alpha = d(t_0)+\beta_{t_0} =  \beta_{t_0},$$ for any $t_0\in\mathbb{S}^1$. Then,
	$$(\beta_t)_p(v) = (i_{M_t}^\ast\alpha)_p(v) =  (\iota_{X_\xi}\omega)_p(v) = \det\left(
	\begin{array}{cc}
	X_\xi^1 & X_\xi^2\\
	v_1 & v_2
	\end{array}
	\right)
	=  - X^2_\xi v_1+X_\xi^1 v_2,$$ for any $p\in M$ and $v = v_1\left.\dfrac{\partial}{\partial x}\right|_p+ v_2\left.\dfrac{\partial}{\partial y}\right|_p\in T_pM$. Therefore $\beta_t = -X_\xi^2dx+X_\xi^1dy$.
	
\end{proof}

\vspace{0.5cm}


Thus, in dimension $3$ we have that every subbundle $\xi\subset T(M\times\mathbb{S}^1)$ transversal to $\mathbb{S}^1$ can be locally written as the kernel of a $1$-form given by
\begin{equation}
\alpha = v_tdt -X_2dx+X_1dy,
\label{alpha1}
\end{equation}
where $X_1$ and $X_2$ are the components of the characteristic vector field of $\xi$ in $M_t$. If $v_t\neq 0$ at every point, then we can rescale the characteristic vector field $X_\xi$ by multiplying it by $v_t$, thus obtaining $\alpha = v_t\widetilde{\alpha}$, where
\begin{equation}
\widetilde{\alpha} = dt-X_2dx+X_1dy
\label{alpha2}
\end{equation}
and $\xi =\ker(\alpha) = \ker(\widetilde{\alpha})$. In this case, we can work~\eqref{alpha2} instead of~\eqref{alpha1}.

\vspace{0.7cm}

\begin{proof}[Proof of Theorem C]
	
	
	Let $S\subset M\times\mathbb{S}^1$ be a surface transversal to $\mathbb{S}^1$. By the Implicit Function Theorem there exist an open set $U\subset M$ such that $S\cap\pi^{-1}(U)$ is the graph of a map $h:U\rightarrow\mathbb{S}^1$ and since $\dim(S) = \dim(M)$, it follows that the restriction of the bundle projection $\pi:M\times\mathbb{S}^1\rightarrow M$ to $S$ is a local diffeomorphism.
	
	
	Shrinking $U$ if necessary, let $(U,x,y)$ be a local symplectic chart in $M$ and define local symplectic coordinates and area form in $S\cap \pi^{-1}(U)$ as the ones induced by the pullback of the coordinates $x,y$ and area form $\omega = dx\wedge dy$ by the bundle projection $\pi$. We shall use the same notation for both coordinates and area forms in $S$ and $M$. By a simple calculation, we see that $\pi^\ast\omega = \omega$.

	Since $\gamma$ is a coboundary with transfer map $\mu$, we know by a part of the proof of Theorem B that $\xi$ is locally the kernel of the $1$-form $\alpha = dt-d\mu$. From the defining equation $\iota_{X_\xi}\omega = i_S^\ast\alpha$ of a vector field representing the characteristic foliation of $\xi$ in $S$, we have
	
	$$-X_2dx+X_1dy = \iota_{X_\xi} = i_S^\ast\alpha = \left(\dfrac{\partial h}{\partial x} - \dfrac{\partial \mu}{\partial x}\right) dx+\left(\dfrac{\partial h}{\partial y} - \dfrac{\partial \mu}{\partial y}\right)dy.$$
	and so $X_1 = -\dfrac{\partial(h-\mu)}{\partial y}$ and $X_2 = \dfrac{\partial (h-\mu)}{\partial x}$.

	

	
	Let $g$ be the riemannian metric in $S$ compatible with the symplectic form $\omega$, $\b{J}$ be the almost complex structure in $S$ and $\grad_0(h-\mu) = \dfrac{\partial (h-\mu)}{\partial x}\dfrac{\partial}{\partial x}+\dfrac{\partial (h-\mu)}{\partial y}\dfrac{\partial}{\partial y}$ in the local symplectic coordinates $x,y$, where $\grad_0(h-\mu)$ is the jacobian matrix of $h-\mu$.
	

	Then, recalling that $\grad_0(h) = \b{G}\grad(h)$, we have $$X_\xi|_{S\cap \pi^{-1}(U)} = -\b{J}_0\grad_0(h-\mu) = \b{J}_0^{-1}\b{G}\grad(h-\mu) =$$ $$= \b{J}^{-1}\grad(h-\mu) = -\b{J}\grad(h-\mu),$$ where $\b{J}_0$ and $\b{G}$ are the tensors associated to $\omega$ and $g$, respectively, and we used the fact that if $(\omega,g,\b{J})$ is a compatible triple in $M$ then $\b{G}^{-1}\b{J}_0 = \b{J}$.
	
\end{proof}

\bigskip

\flushleft

{\bf Vanderlei Horita} (vanderlei.horita\@@unesp.br)\\
Departamento de Matem\'{a}tica, IBILCE/UNESP \\
Rua Crist\'{o}v\~{a}o Colombo 2265\\
15054-000 S. J. Rio Preto, SP, Brazil

\bigskip

\flushleft

{\bf Ricardo Chicalé Lemes} (ricardo.chicale\@@gmail.com)\\
Departamento de Matem\'{a}tica, IBILCE/UNESP \\
Rua Crist\'{o}v\~{a}o Colombo 2265\\
15054-000 S. J. Rio Preto, SP, Brazil

\end{document}